\documentclass[12pt,amscd]{amsart}
\usepackage[all]{xy}
\usepackage{graphicx}
\usepackage{amsmath,amsxtra,amssymb,latexsym, amscd,amsthm,enumerate}
\usepackage{indentfirst}
\usepackage[mathscr]{eucal}
\usepackage{stackrel}
\usepackage{float} 

\newtheorem{thm}{Theorem}[section]
\newtheorem{cor}[thm]{Corollary}
\newtheorem{lem}[thm]{Lemma}

\theoremstyle{definition}

\newtheorem{exam}[thm]{Example}

\DeclareMathOperator{\reg}{reg}

 \DeclareMathOperator{\iin}{in}

\DeclareMathOperator{\pd}{pd}
 \DeclareMathOperator{\rect}{rect}
  \DeclareMathOperator{\Tor}{Tor}

\begin{document}

\title[On the  Betti numbers of  edge ideal of  skew Ferrers graphs]{On the   Betti numbers of  edge ideal of  skew Ferrers graphs}

\author{Do Trong Hoang}

\email{dthoang@math.ac.vn}
\address{Institute of Mathematics, Vietnam Academy of Science and Technology, 18 Hoang Quoc Viet, 10307 Hanoi, Vietnam}

\thanks{}
\date{}
\dedicatory{}
\commby{}
\begin{abstract}   We prove that  $\beta_p(I(G)) = \beta_{p,p+r}(I(G))$ for  skew Ferrers graph $G$, where $p:=\pd(I(G))$ and $r:=\reg(I(G))$. 
 As a consequence, we confirm that Ene, Herzog and Hibi's conjecture is true for the  Betti numbers in the last columm of Betti table.   We also give an explicit formula for the  unique extremal Betti number of binomial edge ideal   for some closed graphs.  
\end{abstract}
\maketitle
\section*{Introduction}

Let $R = k[x_1,\ldots, x_n]$ be a polynomial ring over an arbitrary field $k$. Associated to any homogeneous ideal $I$ of $R$ is a minimal free graded resolution 
  $$0\to \bigoplus_{j} R(-j)^{\beta_{\ell,j}(I)} \to  \bigoplus_{j} R(-j)^{\beta_{\ell-1,j}(I)} \to\cdots\to  \bigoplus_{j} R(-j)^{\beta_{0,j}(I)} \to I  \to 0, $$
where $R(-j)$ denotes the $R$-module obtained by shifting the degrees of $R$ by $j$, and $\ell =\pd(I)$ is the projective dimension of $I$. The number $\beta_{i,j}^R(I)$ (or write $\beta_{i,j}(I)$ if no confusion is caused)  is the {\it $(i,j)$-th graded Betti number}  of $I$ and equals the number of minimal generators of degree $j$ in the $i$-th syzygy module. We have $\beta_{i,j}(I) =   \dim_k\Tor^R_{i+1}(R/I;k)_j$.   The set of graded Betti numbers is represented in terms of a Betti table,  in which the entry at column $i$ and row $j$ is  $\beta_{i,i+j}(I)$.   The  $i$-th total Betti number  of  $I$ is defined by $\beta_i(I) = \sum_{j}\beta_{i,j}(I)$.  The    $\ell$-th column of Betti table  of $I$ is called   {\it  the last column  of Betti table}  of $I$.  The regularity of $I$ is defined by $\reg(I):=\max\{j-i\mid \beta_{i,j}(I)\ne 0\}$.   

Let $G$ be a simple graph on the vertex set $V(G)=\{x_1, \ldots , x_n\}$ and edge set $E(G)$. We associate to the graph $G$ a quadratic squarefree monomial ideal $$I(G) = (x_ix_j  \mid \{x_i,x_j\} \in E(G)) \subseteq R,$$ which is called the {\it edge ideal} of $G$. 
 In \cite{CN}, Corso and Nagel showed that the edge ideal of Ferrers graphs has linear resolution, and furthermore they gave an explicit formula for Betti numbers of this ideals.   After that,   Nagel and  Reiner  (see \cite{NR}) showed that the Betti numbers of the edge ideal of skew Ferrers graphs are independent on the base field $k$.  In this paper, we show that the last Betti number of the edge ideal of skew Ferrers graphs is equal to its  unique extremal Betti number.

\medskip

\noindent  {\bf Theorem \ref{thm0}. } {\it Let  $G$ be  a skew Ferrers graph. Then   $\beta_p(I(G)) = \beta_{p,p+r}(I(G))$, where $p:=\pd(I(G))$ and $r:=\reg(I(G))$.  }   
 
\medskip
 
 Now we assume  a vertex set of $G$ is $V(G)=\{1,\ldots,n\}$. Herzog et al. \cite{HHHKR}; and  Ohtani \cite{O}  independently introduced  
 a {\it binomial edge ideal}, denoted by $J_G$,  associated to  $G$ in polynomial ring  $S:=k[x_1,\ldots,x_n,y_1,\ldots,y_n]$  which is generated by 
  $x_iy_j-x_jy_i$, where  $\{i,j\}\in E(G)$ and $i<j$. It is known that  $\beta_{i,j}(J_G)\leq \beta_{i,j}(\iin(J_G))$ for all $i,j$. This fact implies in particular that  $\pd(J_G)\le \pd(\iin(J_G))$, and $\reg(J_G)\le \reg(\iin(J_G))$. In \cite[Theorem 1.1]{HHHKR},  $J_G$  has a quadratic Gr\"obner basis with respect to the lexicographic order induced  by $x_1>\cdots> x_n>y_1>\cdots> y_n$  if and only if the graph $G$ is  closed   with respect to  the given labeling, in other words, if $G$ satisfies the following condition: whenever $\{i, j\}$ and $\{i, k\}$ are edges of $G$ and either $i <j$, $i< k$ or $i > j$, $i > k$ then $\{j, k\}$ is also an edge of $G$. One calls a graph $G$ {\it closed} if it is closed with respect to some labeling of its vertices.  This concept is  also called {\it PI graph}  (see in \cite{B, H}).  When $G$ is a closed graph, Ene, Herzog and Hibi conjectured in \cite{EHH} that $\beta_{i,j}(J_G)=\beta_{i,j}(\iin(J_G))$ for all $i,j$. This conjecture has been confirmed to be true for Cohen-Macaulay binomial edge ideals in \cite{EHH}, and for closed graphs which consist at most two cliques (see \cite{B}).  Recently, Hern\'an and I (see \cite{DH}) proved the conjecture in some cases when $J_G$ is not Cohen-Macaulay.   Herzog and Rinaldo  also considered the conjecture  for the extremal Betti numbers of binomial edge ideals of  block graphs (see \cite{HR}).  
   
The first result of the  paper is  an affirmation that  the conjecture of Ene, Herzog and Hibi  is true for the  Betti numbers in the last Betti table of the binomial edge ideal of  closed graphs. From that, we obtain that   $\reg(J_G) =\reg(\iin(J_G))$ and $\pd(J_G)=\pd(\iin(J_G))$ for all closed graph $G$. The  statement on the equality of $\reg(J_G)$ and $\reg(\iin(J_G))$ is also proved by Ene and Zarojanu \cite{EZ}.

 \medskip

\noindent  {\bf Theorem \ref{thm1}. } {\it    If $G$ is a closed graph, then   $\reg(J_G) = \reg(\iin(J_G))=:r$, $\pd(J_G) = \pd(\iin(J_G))=:p$, and
 $\beta_p(J_G) = \beta_{p,p+r}(J_G) = \beta_{p}(\iin(J_G))= \beta_{p,p+r}(\iin(J_G))\ne 0.$ 
}
 
 \medskip

The paper is organized as follows. In Section 1, we recall some basic notations and terminology about simplicial complexes. In Section 2, we study non-vanishingness of  Betti numbers of binomial edge ideal of   skew Ferrers graphs in the last column. Section 3 we obtain that the conjecture of Ene, Herzog and Hibi is true for the  Betti numbers in the last column of Betti table.  An explicit formula for the unique extremal Betti number of binomial edge ideal will be given in the last section.

\section{Preliminaries}

A simplicial complex $\Delta$ on the vertex set $V(\Delta):= \{1,\ldots ,n\}$ is a collection of subsets of $V(\Delta)$ such that $F\in \Delta$ whenever $F\subseteq F'$ for some $F'\in \Delta$.   Given any field $k$, we attach to $\Delta$ the {\it Stanley-Reisner} ideal $I_{\Delta}$ of $\Delta$ to be the squarefree monomial ideal
$$I_{\Delta} = (x_{j_1} \cdots x_{j_i} \mid j_1  <\cdots < j_i \ \text{ and } \{j_1,\ldots,j_i\} \notin \Delta) \ \text{ in } R = k[x_1,\ldots,x_n],$$
and the {\it Stanley-Reisner} ring of $\Delta$ to be the quotient ring $k[\Delta] = R/I_{\Delta}$.  This provides a bridge between combinatorics and commutative algebra (see \cite{S}). Then, we say that $\Delta$ is Cohen-Macaulay  over $k$ if $k[\Delta]$ has the same property.  We denote   $\widetilde{H}_j(\Delta;  k)$  is reduced homology group of a simplicial complex $\Delta$ over $k$. The restriction of $\Delta$ to a subset $S$ of $V(\Delta)$ is $\Delta[S] :=\{F\in\Delta \mid F \subseteq S\}$.  A very useful result to compute the graded Betti numbers of the Stanley- Reisner ideal of simplicial complex is the so-called Hochster formula  (c.f. \cite[Theorem 8.1.1]{HH-B}) as follows:
$$ \beta_{i,j}(I_{\Delta}) := \sum_{W\subseteq V(\Delta), |W|=j} \beta_{i,W}(I_{\Delta}),$$  
where  $\beta_{i,W}(I_{\Delta}):= \dim_k \widetilde{H}_{|W|-i-2} (\Delta[W];k)$.  By Hochster formula,  $\beta_{i,j}(I_{\Delta}) \ge \beta_{i,j}(I_{\Delta[S]})$ for all $i,j$ and $S\subseteq V(\Delta)$. Thus we immediately obtain  the following lemma: 

\begin{lem}\label{pd} Let $S\subseteq V(\Delta)$. Then   $\pd(I_{\Delta})\ge \pd(I_{\Delta[S]})$. 
\end{lem} 
 
Let $\Gamma$ and $\Lambda$ be  two simplicial complexes on the disjoint vertex sets  $V(\Gamma)$ and $V(\Lambda)$, respectively. Define the {\it join}  on the vertex $V(\Gamma)\cup V(\Lambda)$ to be $\Gamma*\Lambda=\{\sigma \cup \tau\mid \sigma \in \Gamma, \tau\in \Lambda\}$.   Using K\"unneth formula (c.f. \cite[Proposition 3.2]{BG}), we can  describe the  reduced homology of a join of two simplicial complexes  in terms of the reduced homologies of the factors as follows:  $$\widetilde{H}_{i}(\Gamma*\Lambda;k)\cong \bigoplus_{p+q=i-1} \widetilde{H}_p(\Gamma;k)  \otimes   \widetilde{H}_q(\Lambda;k), \text{ for each }   i. $$  

From this formula, we obtain the following lemma.
 
\begin{lem} \label{lem4} Let   $\Delta=\Delta_1*\ldots*\Delta_m$, where $\Delta_i$ are disjoint subcomplexes of simplicial complex $\Delta$. Let 
$p_i:=\pd(I_{\Delta_i})$,  and   $r_i:=\reg(I_{\Delta_i})$ for $1\le i\le m$.  Then
$$   \beta_{i-1,j}(I_{\Delta})  =  \sum_{\substack{a_1+\ldots+a_m=i,\\ b_1+\ldots+b_m=j}} \prod_{k=1}^m  \beta_{a_k-1,b_k}(I_{\Delta_k}).$$
In particular, $ \pd(I_{\Delta}) = \sum_{i=1}^mp_k +(m-1):=p$,  $\reg(I_{\Delta}) = \sum_{k=1}^m r_k-(m-1):=r$, and  $\beta_{p, p+r}(I_{\Delta}) =\prod_{i=1}^{m} \beta_{p_i,p_i+r_i}(I_{\Delta_i})$.
\end{lem} 
\begin{proof} We   prove  the lemma by   induction on $m$.  If  $m=1$, there is nothing to prove.  Now, we  assume that  $m\ge 2$. Let $\Gamma=\Delta_1*\ldots*\Delta_{m-1}$ and $\Lambda=\Delta_{m}$.  By the induction hypothesis, we have 
\begin{eqnarray} \label{Eqna}
 \beta_{s-1,u}(I_{\Gamma})  =  
\sum_{\substack{a_1+\ldots+a_{m-1}=s,\\ b_1+\ldots+b_{m-1}=u}}  \prod_{k=1}^{m-1}  \beta_{a_k-1,b_k}(I_{\Delta_k}).  
\end{eqnarray}   
 By Hochster formula,      $$\beta_{i-1,j}(I_{\Delta}) = \sum_{W\subseteq V(\Delta), |W|=j} \dim_k \widetilde{H}_{j-i-1} (\Delta[W];k).$$ 
For each $W\subseteq V(\Delta)$,  we have $\Delta[W] =\Gamma[W_1] * \Lambda[W_2]$, where $W_1:=W\cap V(\Gamma)$ and $ W_2:=W\cap V(\Lambda)$. Using K\"unneth formula, we  obtain   that 
\begin{eqnarray*}
 \beta_{i-1,j}(I_{\Delta}) &=&  \sum_{\substack{W\subseteq V(\Delta), \\  |W|=j}} \sum_{p+q=j-i-2} \dim_k \widetilde{H}_p(\Gamma[W_1];k)  \dim_k \widetilde{H}_q(\Lambda[W_2];k) \\
 &=&  \sum_{ \substack{W_1\subseteq V(\Gamma), W_2\subseteq V(\Lambda),\\  |W_1|+|W_2|=j}} \sum_{p+q=j-i-2} \dim_k \widetilde{H}_p(\Gamma[W_1];k)  \dim_k \widetilde{H}_q(\Lambda[W_2];k).
\end{eqnarray*}
Set $u:=|W_1|$, $b_m:=|W_2|$, $s:=u-p-1$ and $a_m:=b_m-q-1$. Thus  $s+a_m=i$ and  we get  
$$ \beta_{i-1,j}(I_{\Delta})  = \sum_{\substack{s+a_m=i,\\  u+b_m=j}}    \beta_{s-1,u}(I_{\Gamma}) \beta_{a_m-1,b_m}(I_{\Lambda}).$$
Using (\ref{Eqna}) for the  above formula,  we  imply  that 
 $$ \beta_{i-1,j}(I_{\Delta})  =\sum_{\substack{a_1+\ldots+a_{m}=i,\\ b_1+\ldots+b_{m}=j}}  \prod_{k=1}^m  \beta_{a_k-1,b_k}(I_{\Delta_k}).$$
From the above formula, we imply  the last statements of lemma.  
\end{proof}

\section{Betti numbers of edge ideal of skew Ferrers graphs} 

In this section, we will study non-vanishingness of the Betti numbers of edge ideal of skew Ferrers graphs in the last columm of the Betti table.  In order to obtain these results, we  recall a rectangular decomposition for skew Ferrers diagram   (c.f.  \cite[Section 2.4]{NR}).  First, we  define a  {\it Ferrers diagram} $D_{X,Y}$ with  $\lambda=(\lambda_1=m\ge \ldots\ge \lambda_n)$ on $X=\{x_1,\ldots,x_n\}$ and $Y=\{y_1,\ldots,y_m\}$ is an array of cells doubly indexed by pairs $(x_i, y_j)$ with $1 \le i \le  n$, $m+1-\lambda_i \le j\le m$. The difference between two Ferrers diagrams is called a  {\it skew Ferrers diagram}.  On the other hand, the skew Ferrers diagram $D_{X,Y}$ on $(X,Y)$, where  $X=\{x_1,\ldots, x_n\}$ and $Y=\{y_1,\ldots, y_m\}$,  is defined by two non-increasing sequences of integers, $\lambda=(\lambda_1=m\ge \cdots \ge \lambda_n)$ and $\mu=(\mu_1\ge \ldots\ge \mu_n)$ and $\lambda_i\ge \mu_i$ for all $i$, and having $\lambda_i-\mu_i$ cells in row $i$, namely
 $\{(x_i,y_j)\mid 1\le i\le n, m+1-\lambda_i\le j \le  m-\mu_i\}$ (see \cite{Ma}).   A skew Ferrers diagram such that $\mu_i=0$ for all $i=1,\ldots,n$ is a Ferrers diagram.  

A bipartite graph $G$ on two distinct vertex sets $X=\{x_1,\ldots,x_n\}$ and $Y=\{y_1,\ldots,y_m\}$ corresponding  to a  Ferrers diagram (resp. skew Ferrers diagram)  is called a   {\it Ferrers graph} (resp. {\it skew Ferrers graph}) if $\{x_i,y_j\}$ is an edge of $G$ whenever  $(x_i,y_j)$ is a  cell in  Ferrers diagram (resp. skew Ferrers diagram).  In \cite{CN},  Corso and Nagel  obtained  the irredundant primary decomposition and gave an explicit  formula for the Betti numbers of edge ideal of Ferrers graphs. 

\begin{lem} {\rm \cite[Theorem 2.1]{CN}} \label{CorNag} Let $G$ be a Ferrers graph.  Then the minimal $\mathbb Z$-graded free resolution of $I(G)$ is $2$-linear  with $i$-th Betti number given by
$$\beta_{i}(I(G)) = \binom{\lambda_1}{i+1} + \binom{\lambda_2+1}{i+1}+\ldots+ \binom{\lambda_n+n-1}{i+1}-\binom{n}{i+2},$$
where $1\le i\le \pd(I(G)) = \max_{1\le j\le n}\{\lambda_j + j - 2\}.$ 
\end{lem}

For skew Ferrers graphs, it is not easy to give an explicit formula for Betti numbers of edge ideals of skew Ferrers graphs (see \cite{DH}). However, Nagel and Reiner \cite[Definition 2.9]{NR} gave a  rectangular decomposition to analyze the homotopy type of the associated the  independence complexes of skew Ferrers graphs.   A  rectangular decomposition  of skew Ferrers diagram $D_{X,Y}$ with   $X=\{x_1, \cdots,x_n\}$ and $Y=\{y_1,\cdots,y_m\}$   is a  partition 
 $$D_{X,Y} = D_{(x_{i_1},y_{j_1})} \sqcup \ldots \sqcup D_{(x_{j_r},y_{j_r})},$$  
where all  $D_{(x_{i_k},y_{j_k})}$  are  defined  inductively as follows:  
 
 \medskip
 
 {\it Step 1:} Choose a top cell $(x_{i_1},y_{j_1}) := (x_1,y_1)$.  We denote    $D_{(x_{i_1},y_{j_1})}$ contains all cells $(x_k,y_l)$ in $D_{X,Y}$ such that either $(x_k,y_{j_1})$ or $(x_{i_1},y_{l})$ is cell  in $D_{X,Y}$.  
   
 \medskip
  
{\it  Step 2:} We set   
\begin{eqnarray*}
X'&:=&\{k\mid k\ge i_1 \text{ and } (x_k,y_{j_1}) \text{ is a cell of } D_{(x_{i_1},y_{j_1})} \}, \\
Y'&:=& \{l\mid l\ge j_1 \text{ and } (x_{i_1}, y_l) \text{ is a cell of } D_{(x_{i_1},y_{j_1})} \}.
\end{eqnarray*} 
We set  $a:=|X'|$ and $b:=|Y'|$. Then $X'=\{i_1,\ldots,i_1+a-1\}$ and $Y'=\{j_1,\ldots,j_1+b-1\}$. Thus, $X\backslash X'=\{i_1+a,\ldots, n\}$ and $Y\backslash Y'=\{j_1+b,\ldots,m\}$. We denote  $X''$ (resp. $Y''$) to be a set of all $k\ge i_1+a$ (resp. $l\ge j_1+b$) such that  $k$-row (resp. $l$-column) of the diagram $D_{X\backslash X', Y\backslash Y'}$ doesn't contain any cell.  We call $X''$ (resp. $Y''$) is {\it empty rectangle} if $X''\ne \emptyset$ (resp. $Y''\ne \emptyset$).  If $X\backslash (X'\cup X'') \ne \emptyset$ or $Y\backslash (Y'\cup Y'')\ne \emptyset$, then we  repeat step 1 for  skew Ferrers diagram $D_{X\backslash (X'\cup X''), Y\backslash (Y' \cup Y'')}$.  

\medskip
 
Finally, we get the rectangular decomposition of  $D_{X,Y}$ as above.  We denote   $\rect(D_{X,Y}):=r$ is called {\it rectangularity number} of $D_{X,Y}$. A skew Ferrers diagram  $D_{X,Y}$  is  called {\it spherical}  if in its rectangular decomposition it has no empty rectangles.  
 
 \begin{exam}  The rectangular decomposition of  a skew Ferrers diagram $D_{X,Y}$   with  $\mu=(4,4,2,2,2,1,0)$ and $\lambda=(7,6,6,5,4,3,2)$ is  
 
$$D_{X,Y} = D_{(x_1,y_1)} \sqcup D_{(x_3,y_4)} \sqcup D_{(x_6,y_6)}, \qquad\qquad  \begin{array}{ccccccccc}
       & y_7  &  y_6  &  y_5  &  y_4  &  y_3  &  y_2  &  y_1  \\
x_1 & &   &  &    & \times & \times  &  \times  \\
x_2 &&  &     &     &  \times & \times  &    \\
x_3&  &  &   \times &  \times &  \times &    &     \\
x_4&   &    &    \times &  \times  &  &    &    \\
x_5&   &     &  \times    &    \times &    &    &    \\
x_6&    &   \times   &   \times &     &    &    &    \\
x_7&  \times   &    \times &     &    &     &   &     \\
\end{array} $$
where  $D_{(x_1,y_1)} = \{(x_1,y_1), (x_1,y_2), (x_1,y_3), (x_2,y_2), (x_2,y_3), (x_3,y_3) \}$,  $D_{(x_6,y_6)} = \{(x_6,y_6),$\linebreak 
$(x_7,y_6), (x_7,y_7) \}$, and  $D_{(x_3,y_4)} = \{(x_3,y_4), (x_3,y_5), (x_4,y_4),(x_4,y_5), (x_5,y_5), (x_5,y_6),  (x_6,y_5) \}$.  Then $\rect(D_{X,Y}) = 3$ and $\{x_2\}$ and $\{y_7\}$ are empty rectangles, but    $D_{X\backslash \{x_2\},Y\backslash \{y_7\}}$ is a  spherical. 
 \end{exam}

\begin{lem} {\rm \cite[Corollary 2.15, Theorem 2.23 and Proposition 2.25]{NR}} \label{lem1} Let $G$ be a skew Ferrers graph with vertex set $X\sqcup Y$. Then 
\begin{enumerate}
\item $\beta_{i, X'\sqcup Y'}(I(G)) = \begin{cases}   1, & \text{ if $D_{X',Y'}$ is spherical with $\rect(D_{X',Y'})=|X'\cup Y'| -i - 1$,}\\
0 & \text{ otherwise,} \end{cases}$ 
 where  $X'\subseteq X$, $Y'\subseteq Y$.  \\
\item $\pd(I(G)) = \max\{|X'\cup Y'| -\rect(D_{X',Y'})-1\}$,  where the maximum runs over all subsets $X'\subseteq X$, $Y'\subseteq Y$  for which  $D_{X',Y'}$ is spherical. 
\item $\reg(I(G)) = \rect(D_{X,Y})+1.$ 
\end{enumerate}
\end{lem} 

Let $G$ be a simple graph. For a subset $S$ of $V(G)$ we denote by $G[S]$  the induced subgraph of $G$ on the vertex set $S$; and denote $G\backslash S$ by $G[V\backslash S]$.   A matching in a graph is a set of edges, no two of which meet a common vertex. An {\it induced matching} $M$ in a graph $G$ is a matching where no two edges of $M$ are adjacented by an edge of $G$.  The maximum size of an induced matching in $G$ is denoted $\nu(G)$. By \cite[Theorem 4.5]{DH},   $\rect(D_{X,Y}) = \nu(G)$   
for  any skew Ferrers graph $G$.

 \begin{lem} \label{lem2} Let $G$ be a skew Ferrers graph with vertex set $X\sqcup Y$. If there exists subsets $X_1\subseteq X$, $Y_1\subseteq Y$ such that  $D_{X_1,Y_1}$ is spherical and    $\pd(I(G))=|X_1\cup Y_1| -\rect(D_{X_1,Y_1})-1,$   then $\rect(D_{X_1,Y_1}) = \rect(D_{X,Y})$. 
 \end{lem} 
 \begin{proof}  Assume on the contrary that  $\rect(D_{X,Y}) \ne \rect(D_{X_1,Y_1})$.   By \cite[Lemma 2.24]{NR}, we may assume that   $\rect(D_{X,Y})>\rect(D_{X_1,Y_1})$. Then  $X_1\ne X$ or $Y_1\ne Y$. By   the rectangular decomposition of $D_{X_1,Y_1}$, we set $M:=\{\{x_{i_1},y_{j_1}\}, \ldots, \{x_{i_{r_1}}, y_{j_{r_1}}\}\}$ is a  maximal induced matching of  $G[X_1\cup Y_1]$, where  $i_1<\ldots<i_{r_1}$, $j_1<\ldots<j_{r_1}$ and $r_1:= \rect(D_{X_1,Y_1})$.   If  $(x_u,y_v)$ is a cell in $D_{X,Y} - D_{X_1,Y_1}$ such that $M\cup \{x_u,y_v\}$ is a maximal induced matching of $G$. Then $D_{X_1\cup \{x_{u}\},Y_1\cup \{y_v\}}$ is also spherical and $\rect(D_{X_1\cup \{x_u\},Y_1\cup \{y_v\}}) = r_1 +1$. However,  $$|(X_1\cup \{x_u\}) \cup (Y_1\cup \{y_v\})| -\rect(D_{X_1\cup \{x_u\},Y_1\cup \{y_v\}}) -1 = |X_1 \cup Y_1| -r_1,$$
 a contradiction thanks to Lemma \ref{lem1} (2).   Therefore, each cell $(x_u,y_v)$  in $D_{X,Y}-D_{X_1,Y_1}$, $M\cup\{x_u,y_v\}$ is not  an induced matching of $G$. Without loss of generality, we assume that $x_u\notin X_1\cup Y_1$ and  $x_uy_{j_t}\in E(G)$ for $1\le t\le r_1$.    Thus, $\rect(D_{X_1\cup\{x_u\}, Y_1}) =r_1$ and  $ i_{t-1}<u< i_{t+1}$, $u\ne i_t$.    
 By the rectangular decomposition of $D_{X_1\cup \{x_u\},Y_1}$, we reduce two following cases: 
$$\begin{array}{lccc}
                 &y_{j_{t+1}} & y_{j_{t}} &y_{j_{t-1}} \\
x_{i_{t-1}}&  &   &\times \\
x_{i_{t}}   &   & \times &  \\
x_{u}        &  & \times &  \\
x_{i_{t+1}}   &\times &   &  \\
\end{array} 
\qquad \text{ or }\qquad 
\begin{array}{lccc}
                 &y_{j_{t+1}} & y_{j_{t}} &y_{j_{t-1}} \\
x_{i_{t-1}}& &   &\times \\
x_{u}        &  & \times &  \\
x_{i_{t}}   &  & \times &  \\
x_{i_{t+1}}   &\times &   &  \\
\end{array} 
 $$ 
 Thus,    $D_{X_1\cup\{x_u\}, Y_1}$ is also spherical, and so
$$|(X_1\cup \{x_u\}) \cup Y_1| -\rect(D_{X_1\cup \{x_u\}, Y_1}) -1 = |X_1\cup Y_1|-r_1,$$
 a contradiction thanks to Lemma \ref{lem1} (2).    Therefore, we conclude that $\rect(D_{X,Y}) = \rect(D_{X_1,Y_1})$, as required.
 \end{proof}

\begin{thm} \label{thm0}  Let  $G$ be  a skew Ferrers graph. Then   $\beta_p(I(G)) = \beta_{p,p+r}(I(G))$, where $p:=\pd(I(G))$ and $r:=\reg(I(G))$. 
\end{thm} 
\begin{proof} By   Lemma \ref{lem1} (3),  we have  $r=\rect(D_{X,Y}) +1$ .   For each $1\le i\le r-1$, we assume $\beta_{p,p+r-i}(I(G))\ne 0$. Since  
$$\beta_{p,p+r-i}(I(G)) = \sum_{V'\subseteq V(G) \text{ and } |V'| = p+r-i} \beta_{p,V'}(I(G)),$$
so  there exists $X'\subseteq X$ and  $Y'\subseteq Y$ such that $|X'\cup Y'|=p+r-i$  and   $\beta_{p, X'\sqcup Y'}(I(G)) \ne 0$. By Lemma \ref{lem1} (1),  we  have  $\beta_{p, X'\sqcup Y'}(I(G)) =1$,  $D_{X',Y'}$ is spherical, and $$p=|X'\cup Y'| - \rect(D_{X',Y'})-1.$$ By Lemma \ref{lem2}, $ \rect(D_{X',Y'}) =  \rect(D_{X,Y})$, and thus, 
$p+r = |X'\cup Y'|$, a contradiction. We conclude that $\beta_{p,p+r-i}(I(G))=0$ for all $1\le i\le r-1$.  Therefore, $\beta_p(I(G)) = \beta_{p,p+r}(I(G))$, as required. 
\end{proof} 
 
\begin{cor}\label{cor}  Let $G$  be a skew Ferrers graph with $p:=\pd(I(G))$ and $r:=\reg(I(G))$.  Then $p+2 \ge r$, and moreover if the equality happens, then   $\beta_{p}(I(G)) =\beta_{p,p+r}(I(G))$ is a number of  induced subgraphs of $G$ consisting of $\nu(G)$ disjoint edges.  
\end{cor} 
\begin{proof}   By Lemma \ref{lem1} (3),    $r=\nu(G)+1$. Let $W$ be a set of vertices of a maximum  induced matching of $G$.  Then, $I(G[W])$  is a complete intersection ideal. Thus, by Lemma \ref{pd},  we have $\pd(I(G))\ge \pd(I(G[W])) = \nu(G)-1$. Hence $p+2\ge \nu(G)+1=r$. 

On the other hand, if $p+2=r$,  then   $p+1=\nu(G)$. By Theorem \ref{thm0} and   \cite[Lemma 2.2]{Ka},   $\beta_{p}(I(G)) =\beta_{p,p+r}(I(G))= \beta_{p,2(p+1)} (I(G))$   is number of  induced subgraphs of $G$ consisting of $\nu(G)$ disjoint edges.  
\end{proof}

\section{Application to binomial edge ideals of closed graphs} 
In \cite[p. 67]{EHH}, Ene, Herzog and Hibi gave a conjecture that the graded Betti numbers of $J_G$ and $\iin(J_G)$ coincide for all closed graphs.  This section is aimed at  proving this conjecture  for the Betti numbers in the last column of their Betti table.    If $G$ is a closed graph without cut vertices,  the initial ideal of binomial edge ideal   $\iin(J_G)$ of $G$  is generated by squarefree monomials of degree two (see \cite[Theorem 1.1]{HHHKR}). We set a nontrivial connected graph $H$,  which is called {\it initial-closed} graph, corresponding  to $\iin(J_G)$. In \cite{DH}, Hern\'an and I studied the structure of this initial-closed graphs, and we realize that   one is a special skew Ferrers graph. 
 
Let $G$ be a closed graph without closed graph  on vertex set $V(G)=\{1,\ldots,n\}$. Define $N^{>}_{G}(i):=\{j\in V(G)\mid i<j \text{ and } \{i,j\}\in E(G)\}$ and $\deg_G^{>}(i):=|N^{>}_{G}(i)|$.  We associate with $G$ a vector $\mu(G)=(\mu_1,\ldots,\mu_{n})\in \mathbb N^{n}$, where $\mu_j=n-j-\deg_G^{>}(j)$ for all $1\le j\le n$. Then the initial-closed graph $H$ is a bipartite graph with bipartition $(X,Y)$, where $X=\{x_1,\ldots,x_{n-1}\}$ and $Y=\{y_1,\ldots,y_{n-1}\}$, and $\mu(H)=(\mu_1,\ldots,\mu_{n-1})\in \mathbb N^{n-1}$. By \cite[p. 33]{DH},  $H$ is a skew Ferrers graph with $\lambda_i=n-i$ and $\mu_i\le n-2-i$ for $1\le i\le n-3$, $\mu_{n-2}=\mu_{n-1}=0$.   The following theorem confirms that conjecture of Ene, Herzog and Hibi holds for  Betti numbers  in the last column of Betti table. 
  
 \begin{thm} \label{thm1}
 If $G$ is a closed graph, then   $\reg(J_G) = \reg(\iin(J_G))=:r$, $\pd(J_G) = \pd(\iin(J_G))=:p$, and
 $\beta_p(J_G) = \beta_{p,p+r}(J_G) = \beta_{p}(\iin(J_G))= \beta_{p,p+r}(\iin(J_G))\ne 0.$  
 \end{thm} 
\begin{proof}  If  $G$ is disconnected,  we may assume $G_1,\ldots,G_s$ are connected components of $G$.   It is well-known that the  $(i,j)$-th Betti numbers of  $J_G$  and $\iin(J_G)$ coincide if the $(i,j)$-th Betti numbers of   $J_{G_k}$  and $\iin(J_{G_k})$  coincide for all $k=1,\ldots,s$.  Therefore, we now consider $G$  is a connected graph. Then there exists $m$  $(m\ge 0)$  cut vertices of $G$, say $v_1,\ldots,v_{m}$. We  may write $G$ in the form   $$G=G_1\cup\ldots\cup  G_{m+1},$$  
   where  $G_i$ is a subgraph without  cut vertices of $G$,  and for $1\le i< j\le m+1$ either $G_{i}\cap G_{j}=\emptyset$, or     $G_i\cap G_{j}=\{v_k\}$ for some $k$. By the assumption, $G$ is a closed graph, so  $G_i$ is.   Let $r_i:=\reg(\iin(J_{G_i}))$ and $p_i:=\pd(\iin(J_{G_i}))$, and $p:=\pd(\iin(J_G))$ and $r=\reg(\iin(J_G))$.

Let $H$ (resp. $H_i$) be a nontrivial  graph such that $I(H)=\iin(J_G)$ (resp. $I(H_i)=\iin(J_{G_i})$). Thus, $\beta_{p}(\iin(J_G)) =\beta_{p}(I(H))$ and $\beta_{p,p+r}(\iin(J_G)) =\beta_{p,p+r}(I(H))$.  First, we claim that  $\beta_{p}(\iin(J_G)) = \beta_{p,p+r}(\iin(J_G)) \ne 0$. In order to prove this, we  need to prove   $\beta_{p}(I(H)) =  \beta_{p,p+r}(I(H)) \ne 0$. Indeed, for each $1\le i\le m+1$, by \cite[Lemma 2.4]{DH}, $H_i$ is an initial-closed graph correponding to closed graph without cut vertices $G_i$.  By Theorem  \ref{thm0}, we have  $\beta_{p_i}(I(H_i)) = \beta_{p_i,p_i+r_i}(I(H_i))\ne 0$.   Moreover, by \cite[Lemma 2.7]{DH}, $H_1,\ldots, H_{m+1}$ are connected components of $H$, and thus   $H=H_1\sqcup \ldots \sqcup H_{m+1}$.  So
 $p= p_1+\ldots+p_{m+1} + m$,  $r=r_1+\ldots+r_{m+1} - m $ and 
 $$\Delta(H) = \Delta(H_1)*\ldots*\Delta(H_{m+1}).$$
   By Lemma \ref{lem4},  we have  $\beta_{p,p+r} (I(H)) = \prod_{i=1}^{m+1} \beta_{p_i,p_i+r_i}(I(H_i))\ne 0$, and moreover   for $1\le j\le r-1$, 
  \begin{eqnarray*}
    \beta_{p,p+r-j} (I(H)) &=& \sum_{u_1+\ldots+u_{m+1}= p+r-j} \prod_{k=1}^{m+1} \beta_{p_k,u_k}(I(H_k)).
  \end{eqnarray*}
  Since $u_1+\ldots+u_{m+1}=p+r-j$, there exists $1\le \ell \le m +1$ such that $u_{\ell}<p_{\ell}+r_{\ell}$. This implies  that  $\beta_{p_{\ell},u_{\ell}}(I(H_{\ell}))=0$, which means that $\beta_{p,p+r-j}(I(H))=0$.  Hence, $ \beta_{p}(I(H)) = \beta_{p,p+r}(I(H)) \ne 0 $, as claimed.  

On the other hand, since $J_G$ and $\iin(J_G)$ have the same Hilbert polynomial and together with $\beta_{p,p+r}(\iin(J_G))\ne 0$, we obtain    $\beta_{p,p+r}(\iin(J_G))= \beta_{p,p+r}(J_G)$,  $r=\reg(J_G) $ and $p=\pd(J_G)$. Moreover, for each $0\le j\le r-1$, we have  $\beta_{p,p+r-j}(J_G)\le \beta_{p,p+r-j}(\iin(J_G))=0$ which means that $\beta_{p,p+r-j}(J_G) =0$.   Therefore, $\beta_{p}(J_G) =\beta_{p,p+r}(J_G)$, as required. 
  \end{proof} 
 
 \begin{cor}  If $G$ is a closed graph, then     
$$\beta_{p-1,p+r-1}(J_G) = \beta_{p-1,p+r-1}(\iin(J_G)),$$
 where $r:=\reg(J_G)$ and $p:=\pd(J_G)$.   
 \end{cor} 
 \begin{proof} 
 Since  $J_G$ and $\iin(J_G)$ have the same Hilbert polynomial, 
$$\beta_{p-1,p+r-1}(J_G) - \beta_{p,p+r-1}(J_G) = \beta_{p-1,p+r-1}(\iin(J_G)) - \beta_{p,p+r-1}(\iin(J_G)).$$
By Theorem \ref{thm1}, we have $\beta_{p,p+r-1}(\iin(J_G))= \beta_{p,p+r-1}(J_G)=0$. Therefore,  $\beta_{p-1,p+r-1}(J_G) = \beta_{p-1,p+r-1}(\iin(J_G)),$   as required. 
 \end{proof}

\section{The unique extremal Betti numbers of binomial edge ideal  of certain closed graphs}  

To compute the unique extremal Betti numbers of binomial edge ideal  of closed graphs, by the argument of Theorem \ref{thm1},  we can reduce the case  closed graphs without  cut vertices.  Now we let $G$ be a closed graph without cut vertices, and   $H$ be  an initial-closed graph with $\mu(H)=(\mu_1,\ldots,\mu_{n-1})\in \mathbb N^{n-1}$ corresponding to $G$. By  \cite[Theorem 4.5]{DH}, the regularity of  $J_G$ equals three if and only if $\mu_1=\ldots=\mu_s=:\mu \ge 1$ and $s\ge 1$, where $s:=\min\{k-1\mid \mu_k=0\}$.    In this section, we will give an explicit formula for the unique extremal Betti number of  binomial edge ideal of closed graph without cut vertices $G$  whenever  the regularity of  $J_G$ equals three.  From there, we have  an explicit  formula for larger regularity case.

Firstly, we need some technical lemmas:  

\begin{lem}\label{tor} Let $I=J+(x_1,\ldots,x_s)\subseteq S:=k[x_1,\ldots,x_n]$ with $1\le s\le n$,  where $J\subseteq S/(x_1,\ldots,x_s)=:S'$. Then $p:=\pd(I) = \pd(J) +s$, and  $ \beta_{p,j}^S(I) = \beta^{S'}_{p-s,j-s}(J).$ 
\end{lem} 
\begin{proof}   We will prove the lemma by induction on $s$.  If $s=1$,   by \cite[Remark 2.1]{BCR}, we have $$\beta^S_{p,j}(S/I) = \beta_{p,j}^{S_1}(I) +  \beta_{p-1,j-1}^{S_1}(I).$$
Since $\pd(J)=\pd(I)  -1=p-1$, so $\beta_{p,j}^{S_1}(I)=\beta_{p,j}^{S_1}(J)= 0$. Thus, we complete the proof in this case. 

If $s\ge 2$, let   $S_{s-1}:=S/(x_s)$ and $J_{s-1} = J+ (x_1,\ldots,x_{s-1})$. Then $I=J_{s-1}+(x_s)$.   By \cite[Remark 2.1]{BCR}, we have 
$$\beta_{p,j}^S(I)  = \beta^{S_{s-1}}_{p,j}(J_{s-1}) +  \beta^{S_{s-1}}_{p-1,j-1}(J_{s-1}).$$ 
However, since  $\pd(I_{s-1})=p-1$, so $\beta^{S_{s-1}}_{p,j}(J_{s-1})=0$. This means that   $\beta_{p,j}^S(I)  =  \beta^{S_{s-1}}_{p-1,j-1}(J_{s-1})$. By the induction hypothesis,        $\beta_{p,j}^S(I)  =\beta_{p-s,j-s}^{S'}(J)$, as required.
\end{proof} 

For any simple graph $G$, the neighborhood of a vertex $x$ of $G$ is the set $N_G(x) := \{y\in V(G) \mid \{x,y\} \in E(G)\}$.  If $S=\{x\}$,  we write  $G\backslash x$ (resp. $G_x$)  instead of  $G\backslash \{x\}$ (resp.  $G\backslash (N_G(x)\cup \{x\})$).

\begin{lem} \label{tech_betti} Let $x$ be a vertex of $G$ with neighbors $y_1, y_2, \ldots , y_s$.  Let   $S_1:=k[V(G\backslash x)]$ and  $S_2:= k[V(G_x)]$.  Let  $p:=\pd(I(G))$.  For each $j$, 
\begin{enumerate}
\item If $p=\pd((I(G),x))$, then    $\beta_{i}^S((I(G),x)) =0$ for all $i>p$, and $$\beta_{p,j}^S((I(G),x)) = \beta_{p-1,j-1}^{S_1}(I(G\backslash x)).$$
\item   If $p= \pd((I(G):x))$, then  $\beta_{p,j}^S((I(G):x)(-1))  =   \beta_{p-s, j-1-s}^{S_2}(I(G_x))$. 
\end{enumerate}
\end{lem}
\begin{proof} In order to prove the lemma, we first use the following result by \cite[Lemma  3.1]{DHS}: 
\begin{eqnarray*}
I(G):x &=& I(G_x) + (y_{1},\ldots,y_{s}),\\
(I(G),x) &=& I(G\backslash x) + (x).
\end{eqnarray*} 

(1) By the definition of the projective dimension, we get $\beta_{i}^S((I(G),x))=0$ for all $i>p$.   By Lemma \ref{tor}, we have $\beta_{p,j}^S((I(G),x))    = \beta^{S_1}_{p-1,j-1}(I(G\backslash x))$. This   completes  the proof of the assertion. 

\medskip

(2) We have $\beta_{p,j}^S((I(G):x)(-1))    =    \beta_{p,j-1}^S((I(G):x))$. By Lemma \ref{tor},   $\beta_{p,j-1}^S((I(G):x))=   \beta_{p-s,j-1-s}^{S'}(I(G_x))$, where $S':=S/(y_1,\ldots,y_s)$.  Since $I(G_x)$ is an  ideal in polynomial ring $S_2$, we imply that   $\beta_{p-s,j-1-s}^{S'}(I(G_x))=\beta_{p-s,j-1-s}^{S_2}(I(G_x))$. Hence, the assertion is proved. 
 \end{proof}
 
 \begin{lem} \label{tech_lem}
 Let $H$ be an  initial-closed graph with $\mu=(\mu_1,\ldots,\mu_{n-1})\in \mathbb N^{n-1}$, where  $\mu_1=\ldots=\mu_s=1$,   $\mu_{s+1}=\ldots=\mu_{n-1}=0$ and  $s\ge 1$. Then 
 $$\beta_{2n-s-4, 2n-s-1}(I(H)) = s.$$ 
 \end{lem} 
 \begin{proof} Let $p:=2n-4-s$. By \cite[Lemma 3.8]{DH}, we have $\pd(I(H)) = p$.  Note that $n-s\ge 3$.   Since $H_{y_{n-1}}$ is a  Ferrers graph with   $\lambda(H_{y_{n-1}}) = (n-2, \ldots,n-1-s)\in \mathbb N^{s}$.   By Lemma \ref{CorNag},    $\pd(I(H_{y_{n-1}})) = n-3$, $\reg(I(H_{y_{n-1}}))=2$ and  $\beta_{n-3,n-1}^{S'}(I(H_{y_{n-1}}))=\beta_{n-3}^{S'}(I(H_{y_{n-1}})) = s$, where $S':=k[V(H_{y_{n-1}})]$.   Since $N_H(y_{n-1}) = \{x_{s+1},\ldots,x_{n-1}\}$, 
   $$\pd((I(H):y_{n-1}) = \pd(I(H_{y_{n-1}})) + (n-s-1) =  (n-3) + (n-s-1) = p. $$
On the other hand,  $H\backslash \{x_{n-1}, y_{n-1}\}$ is also a Ferrers  graph with $\lambda(H\backslash \{x_{n-1}, y_{n-1}\})=(n-2,\ldots,2,1)\in \mathbb N^{n-2}$. By Lemma \ref{CorNag}, $\pd(I(H\backslash \{x_{n-1}, y_{n-1}\})) =   n-3$.   Thus,  $$\pd((I(H),y_{n-1}))=1+\pd(I(H\backslash  x_{n-1}))= 1+ \pd(I(H\backslash \{x_{n-1},y_{n-1}\}))  =  n-2.$$  Since   $p >   n-2$,   so $\Tor_{i}^S(S/(I(H),y_{n-1});k)=0$ for $i\ge p$. 
  
 From a short exact following sequence: 
$$0\to S/(I(H):y_{n-1})(-1)\to  S/I(H)\to  S/(I(H),y_{n-1}) \to 0, $$  
 we  get  a long exact sequence of $\Tor$-modules.  This implies that   the   following  sequence:   $$0 \to  \Tor^{S}_{p+1}(S/(I(H):y_{n-1})(-1);k)_{p+2} \to \Tor^S_{p+1}(S/I(H);k)_{p+2} \to  0$$
 is exact.    Thus, $\beta_{p,p+3}^S(I(H)) = \beta_{p,p+3}^{S}((I(H):y_{n-1})(-1)) =\beta_{p,p+2}^S((I(H):y_{n-1}))$.   Together with   Lemma \ref{tech_betti}, we get  $\beta_{p,p+3}^S(I(H)) = \beta_{p-(n-s-1),p+2-(n-s-1)}^{S'}(I(H_{y_{n-1}})) =\beta_{n-3,n-1}^{S'}(I(H_{y_{n-1}}))=s$.  
 \end{proof}

 \begin{lem} \label{lem_betti_initial} Let  $H$ be an  initial-closed graph with $\mu(H)=(\mu_1,\ldots,\mu_{s}, 0,\ldots,0)\in \mathbb N^{n-1}$, where $\mu_1=\ldots=\mu_s=:\mu>0$ and $s\ge 1$.  Then 
  $$\beta_{2n-\mu-s-3, 2n-\mu-s} (I(H))= s\mu.$$
 \end{lem} 
\begin{proof}    By \cite[Lemma 3.8]{DH},  $\pd(I(H)) = 2n- \mu-s-3=:p$.  We will prove by  induction on $\mu$.  If $\mu=1$, the lemma is proved by Lemma \ref{tech_lem}. Now we assume that  $\mu\ge 2$.  
Since $H_{y_{n-1}}$ is a  Ferrers graph with   $\lambda(H_{y_{n-1}}) = (n-\mu-1, \ldots,n-\mu-s)\in \mathbb N^{s}$. By Lemma \ref{CorNag},  $\pd(I(H_{y_{n-1}})) = n-\mu-2$, $\reg(I(H_{y_{n-1}})) =2$ and  $\beta_{n-\mu-2,n-\mu}(I(H_{y_{n-1}})) =\beta_{n-\mu}(I(H_{y_{n-1}})) =s$.  Then   
\begin{eqnarray*}
\pd((I(H):y_{n-1}) &=& \pd(I(H_{y_{n-1}})) + (n-s-1) = p,  \\
\reg((I(H):y_{n-1}) &=& \reg(I(H_{y_{n-1}})) =3.
\end{eqnarray*}
Thus,   $\beta_{p-1,p+3}^{S}((I(H):y_{n-1})(-1)) = \beta_{p-1,p+3}^{S}((I(H):y_{n-1})) =0$.  Hence, we conclude that  $\Tor_{p}^S(S/(I(H):y_{n-1})(-1);k)_{p+3}={\bf 0}$.

On the other hand,  $H\backslash \{x_{n-1}, y_{n-1}\}$ is an initial-closed graph with $\mu(H\backslash \{x_{n-1}, y_{n-1}\})=(\mu-1,\ldots,\mu-1,0,\ldots,0)\in \mathbb N^{n-2}$. By \cite[Lemma 3.8]{DH}, $\pd(I(H\backslash \{x_{n-1}, y_{n-1}\})) = 2n-3- \mu-1-s  = p-1$.  Then  $\pd((I(H),y_{n-1}))=1+\pd(I(H\backslash \{x_{n-1},y_{n-1}\}))   =   p. $  Thus,  $\Tor_{p+2}^S(S/(I(H),y_{n-1});k)  = {\bf 0}$.  

From a short  exact sequence of $S$-modules:
\begin{eqnarray*}  
0\to S/(I(H):y_{n-1})(-1) \to S/I(H) \to S/(I(H),y_{n-1})\to 0,
\end{eqnarray*} 
we obtain  a long exact sequence of $\Tor$-modules and thus, the following sequence
\begin{eqnarray*}
0   \to   \Tor^{S}_{p+1}(S/(I(H):y_{n-1})(-1);k)_{p+3} &\to& \Tor^S_{p+1}(S/I(H);k)_{p+3}\\
 &\to& \Tor^{S}_{p+1}(S/(I(H), y_{n-1});k)_{p+3} 
  \to  0, 
\end{eqnarray*} 
is exact.   Therefore, we get  $$\beta_{p,p+3}^S(I(H)) = \beta_{p,p+3}^S((I(H):y_{n-1})(-1)) +\beta_{p,p+3}^S((I(H), y_{n-1}).$$
Let $S_1:=k[V(H\backslash \{x_{n-1},y_{n-1}\})]$ and $S_2:=k[V(H_{y_{n-1}})]$.  By Lemma \ref{tech_betti}, we obtain 
 \begin{eqnarray*}
\beta_{p,p+3}^S(I(H)) &=&    \beta_{n-\mu-2, n-\mu}^{S_2}(H_{y_{n-1}})+ \beta_{p-1,p+2}^{S_1}(I(H\backslash \{x_{n-1},y_{n-1}\})) \\
&=&  s + \beta_{p-1,p+2}^{S_1}(I(H\backslash \{x_{n-1},y_{n-1}\})). 
\end{eqnarray*} 
By the induction  hypothesis,  we have $\beta_{p-1,p+2}^{S_1}(I(H\backslash \{x_{n-1},y_{n-1}\})) = (\mu-1)s$. Hence, 
we conclude that $\beta_{p,p+3}(I(H)) =  s + (\mu-1)s =\mu s$, as required.  
\end{proof} 

\begin{thm} \label{thm2}  Let  $G$  be a closed graph without cut vertices with $\mu(G)=(\mu_1,\ldots,\mu_n)\in \mathbb N^n$. Then 
  $\reg(J_G)=3$ if and only if $\mu_1=\ldots=\mu_s=:\mu\ge 1$ and $ s\ge 1$, where $s:=\min\{k-1\mid \mu_k=0\}$. In particular,  
$$\beta_{p}(J_G) =\beta_{p,p+3}(J_G) = \beta_{p}(\iin(J_G))  =\beta_{p,p+3}(\iin(J_G)) = s\mu,$$
where $p=\pd(J_G) = \pd(\iin(J_G)) = 2n-\mu-s-3$. 
\end{thm} 
 \begin{proof} The first statement is followed by  \cite[Corollary 4.7]{DH}.  Now we prove the second statement. Let $H$ be an initial-closed graph   corresponding to closed graph without cut vertices $G$.  Then $\beta_{p,p+3}(\iin(J_G)) = \beta_{p,p+3}(I(H))$.  By the assumption,  we have $\mu(H)=(\mu_1,\ldots,\mu_{n-1})\in \mathbb N^{n-1}$, where   $\mu_1=\ldots=\mu_s=:\mu\ge 1$ and $  s\ge 1$. By Lemma \ref{lem_betti_initial}, $\beta_{p,p+3}(I(H)) = s\mu$. From this, the theorem is proved  by  Theorem \ref{thm2}.   
  \end{proof}
  
 \begin{thm} 
Let $G$ be a connected closed graph with $m$ cut vertices, say  $v_1,\ldots,v_{m}$. Then  $G$ is written in the form $$G=G_1\cup\ldots\cup  G_{m+1},$$ where  $G_i\cap G_{i+1}=\{v_i\}$ and $G_i\cap G_j = \emptyset$ for $i=1,\ldots,m$ and $i\ne j\ne i+1$.  Let   $\mu(G_i)= (\mu_{i1},\ldots,\mu_{in_i})$   and $s_i:=\min\{k-1\mid \mu_{ik} =0\}$, where $n_i=|V(G_i)|$.  If      
 $\mu_{i1}=\ldots=\mu_{is_i}=:\mu_i\ge 1$ and $s_i\ge 1$ for all $i$,   then    
$$\beta_p(J_G) = \beta_{p,p+r}(J_G) = \beta_p(\iin(J_G)) = \beta_{p,p+r}(\iin(J_G)) = \prod_{i=1}^{m+1} s_i\mu_i,$$
where $p:=\pd(J_G) =2n - 3 - \sum_{i=1}^{m+1} (\mu_i+s_i)$ and $r:=\reg(J_G)=2m+3$. 
 \end{thm} 
 \begin{proof} By Theorem \ref{thm1}, we only need to prove $\beta_{p,p+r}(\iin(J_G)) = \prod_{i=1}^{m+1} s_i\mu_i$, where 
 $p=\pd(\iin(J_G)) =2n - 3 - \sum_{i=1}^{m+1} (\mu_i+s_i)$ and $r=\reg(\iin(J_G))=2m+3$.  First,  we call  $H$ is a nontrivial graph such that $I(H) =\iin(J_G)$ and  $H_i$ is  an initial-closed graph correspoding to the closed graph without cut vertices $G_i$ for  $1\le i\le m+1$.  Then $n=\sum_{i=1}^{m+1} n_i - m$.  By the assumption, we have  $p_i:=\pd(I(H_i)) = \pd(\iin(J_{G_i})) = 2n_i-\mu_i-s_i-3$ and $r_i:=\reg(I(H_i)) =\reg(\iin(J_{G_i}))= 3$.  Thus,   
  $p=\pd(I(H))=p_1+\ldots+p_{m+1}+ m = 2n - 3 - \sum_{i=1}^{m+1} (\mu_i+s_i)$,  $r=\reg(I(H))=r_1+\ldots+r_{m+1}- m =2m+3$. By  \cite[Lemma 2.7]{DH}, $H= H_1\sqcup \ldots\sqcup H_{m+1}$.  By Lemma \ref{lem4}, 
 $$\beta_{p,p+r}(I(H)) = \prod_{i=1}^{m+1} \beta_{p_i, p_i+r_i}(I(H_i)).$$
  Together with  Theorem \ref{thm2},        $\beta_{p,p+r}(\iin(J_G)) =\beta_{p,p+r}(I(H))= \prod_{i=1}^{m+1} s_i\mu_i$,   as required. 
 \end{proof}

\subsection*{Acknowledgment}  The author is  grateful to Hern\'an de Alba  for a discussion. I am partially supported by the NAFOSTED (Vietnam) under grant number  101.04-2015.02.

  \end{document}